\documentclass{amsproc}

\newtheorem{theorem}{Theorem}[section]
\newtheorem{lemma}[theorem]{Lemma}

\theoremstyle{definition}

\newtheorem{hypothesis}[theorem]{Hypothesis}
\newtheorem{corollary}[theorem]{Corollary}

\theoremstyle{remark}
\newtheorem{remark}[theorem]{Remark}

\numberwithin{equation}{section}

\newcommand{\bbR}{\mathbb R}

\newcommand{\bbD}{\mathbb D}
\newcommand{\bbC}{\mathbb C}
\newcommand{\bbN}{\mathbb N}

\renewcommand{\epsilon}{\varepsilon}
\newcommand{\be}{\begin{equation}}
\newcommand{\ee}{\end{equation}}

\newcommand{\R}{\mathbb{R}}

\newcommand{\cB}{{\mathcal B}}

\newcommand{\cH}{{\mathcal H}}

\newcommand{\cM}{{\mathcal M}}

\newcommand{\cU}{{\mathcal U}}

\renewcommand{\Im}{{\ensuremath{\mathrm{Im}}}}

\newcommand{\fm}{\mathfrak{m}}

\DeclareMathOperator{\Ran}{\mathrm{Ran}}
\DeclareMathOperator{\Ker}{\mathrm{Ker}}

\newcommand{\linspan}{\mathrm{lin\ span}}

\newcommand{\Dom}{\mathrm{Dom}}
\newcommand{\dom}{\mathrm{Dom}}

\begin{document}

\title{On the Weyl-Titchmarsh and Liv\v{s}ic  functions}

\author{K. A. Makarov}
\address{Department of Mathematics, University of Missouri, Columbia, MO 63211, USA}
\email{makarovk@missouri.edu}

\author{E. Tsekanovski\u{i} }
\address{
 Department of Mathematics, Niagara University, P.O. Box 2044,
NY  14109, USA } 
\email{tsekanov@niagara.edu}

\dedicatory
{Dedicated with great pleasure to our friend and colleague Fritz
Gesztesy on the occasion of his 60th birthday anniversary}



\begin{abstract} We establish a mutual
relationship between
main  analytic objects for the dissipative extension theory of 
a symmetric operator $\dot A$ 
with deficiency indices $(1,1)$. In particular, 
we introduce the Weyl-Titchmarsh function $\cM$ 
of a maximal dissipative extension
$\widehat A$ of the symmetric operator $\dot A$. 
Given   a
reference self-adjoint extension  $A$ of
 $\dot A$, we introduce  a von Neumann parameter
 $\kappa$, $|\kappa|<1$, characterizing the domain 
of the dissipative extension $\widehat A$ against $\Dom (A)$ and show
that the pair $(\kappa, \cM)$  is a complete  unitary 
  invariant    of the triple $(\dot A, A, 
\widehat A)$, unless $\kappa=0$.
As a by-product of our considerations we obtain a relevant 
 functional model for
 a  dissipative
operator and get an analog of the formula of Krein for its resolvent.

\end{abstract}

 \maketitle

\section{Introduction}

In 1946 M. Liv\v{s}ic obtained the following remarkable result
\cite[a part of Theorem 13]{L} (for a textbook
exposition  see \cite{AkG}).
\begin{theorem}[\cite{L}]\label{thm1}
Suppose that $\dot A$ and $\dot B$
 are closed, prime\footnote{A symmetric operator $\dot A$ is prime
if there does not exist a subspace invariant under $\dot A$ such that
the restriction of $\dot A$ to this subspace is self-adjoint.},
 densely defined symmetric operators with deficiency indices $(1,1)$.
Then $\dot A$ and $\dot B$ are unitarily equivalent if, and only
if, for an appropriate choice of normalized deficiency elements
$$g_\pm\in \Ker( (\dot A)^*\mp iI) \quad \text{and }\quad
f_\pm\in \Ker( (\dot B)^*\mp iI)$$
the following equality
$$
\frac{(g_z, g_-)}{(g_z, g_+)}=\frac{(f_z, f_-)}{(f_z, f_+)}, \quad z\in \bbC_+,
$$
holds, where $ g_z\ne 0$
and $ f_z\ne 0$
are arbitrary  deficiency elements of the symmetric operators $\dot A$
 and $\dot B$,
$$ g_z\in \Ker( (\dot A)^*-z I) \quad \text{and}\quad 
 f_z\in \Ker( (\dot B)^*-z I).$$
\end{theorem}

Liv\v{s}ic suggested to call the function
\begin{equation}\label{charsym}
s(z)=\frac{z-i}{z+i}\cdot \frac{(g_z, g_-)}{(g_z, g_+)}, \quad z\in \bbC_+,
\end{equation}
{\it the characteristic function} of the symmetric operator $\dot A$.

Theorem \ref{thm1} identifies the function $s(z)$ (modulo $z$-independent unimodular factor) with
a complete unitary invariant  of a prime
 symmetric operator with deficiency indices $(1,1)$
that  determines the operator uniquely up to unitary equivalence.

Liv\v{s}ic also gave  a criterion \cite[Theorem 15]{L} (also see \cite{AkG})
 for a contractive analytic mapping
 from the upper half-plane  $\bbC_+$ to the unit disk
$\bbD$ to  be the characteristic function of a densely defined
symmetric operator  with deficiency indices $(1,1)$.

\begin{theorem}[\cite{L}]\label{thm12}\label{nach1} For an
 analytic mapping $s$ from the upper half-plane to the unit disk
to be the characteristic function of a densely defined
symmetric operator  with deficiency indices $(1,1)$
it is necessary and sufficient that
\begin{equation}\label{vsea0}
s(i)=0\quad \text{and}\quad \lim_{z\to \infty}
z(s(z)-e^{2i\alpha})=\infty \quad \text{for all} \quad  \alpha\in
[0, \pi),
\end{equation}
$$
0< \varepsilon \le \text{arg} (z)\le \pi -\varepsilon.
$$
\end{theorem}

In the same article, Liv\v{s}ic
put forward a concept of a characteristic function of
a quasi-self-adjoint dissipative extension
of a symmetric operator with deficiency indices $(1,1)$.

Let us recall Liv\v{s}ic's construction.

Suppose that $\dot A$ is a symmetric operator with deficiency indices $(1,1)$
and that
$g_\pm$ are its normalized deficiency elements,
$$
g_\pm \in \Ker ((\dot A)^*\mp i I), \quad \|g_\pm\|=1.
$$

Suppose that $\widehat A\ne (\widehat A)^*$ is a  maximal dissipative
extension of $\dot A$, 
$$\Im(\widehat Af,f)\ge 0, \quad f\in \Dom(\widehat A)
.$$
Since $\dot A$ is symmetric, its dissipative extension $\widehat A$ 
is automatically quasi-self-adjoint \cite{Ph}, \cite{St68}, 
that  is, 
$$
\dot A \subset\widehat A\subset (\dot A)^*,
$$
and hence 
\begin{equation}\label{parpar}
g_+-\kappa g_-\in \Dom
 (\widehat A)\quad \text{for some }
|\kappa|<1.
\end{equation}
Based on the  parameterization \eqref{parpar} of the domain of the
 extension $\widehat A$, which is an analog of the von Neumann
formulae\footnote{Throughout this paper  $\kappa$ will be
called the von Neumann extension parameter.},
Liv\v{s}ic suggested to call  the M\"obius transformation
\begin{equation}\label{ch12}
S(z)=\frac{s(z)-\kappa} {\overline{ \kappa }\,s(z)-1}, \quad z\in \bbC_+,
\end{equation}
where $s$ is given  by \eqref{charsym},
the characteristic function of the dissipative extension $\widehat A$
\cite{L}.

A culminating point of Liv\v{s}ic's considerations was the discovery of the
 following result \cite[the remaining  part of Theorem 13]{L}.

\begin{theorem}[\cite{L}]\label{thmdiss} Suppose that $\dot A$ and $\dot B$
 are closed prime densely defined symmetric operators with deficiency
indices $(1,1)$. Assume, in addition,  that
$\widehat A$ and $\widehat B$ are their maximal dissipative
extensions, respectively ($\widehat A\ne (\widehat A)^*$, 
$\widehat B\ne (\widehat B)^*$).

Then, $\widehat A$ and $\widehat B$
are unitarily equivalent if, and only if, the corresponding
characteristic functions
coincide up to a unimodular constant factor.
\end{theorem}

In 1965  Donoghue \cite{D}
 introduced a concept of the Weyl-Titchmarsh function $M(\dot A, A)$
associated with a pair $(\dot A, A)$
by
$$M(\dot A, A)(z)=
((Az+I)(A-zI)^{-1}g_+,g_+), \quad z\in \bbC_+,
$$
$$g_+\in \Ker( (\dot A)^*-iI),\quad \|g_+\|=1,
$$where $\dot A $ is
 a symmetric operator with deficiency indices $(1,1)$, 
$\text{def}(\dot A)=(1,1)$,
 and $A$ is
its self-adjoint extension\footnote{The concept of the Weyl--Titchmarsh function in the general case where 
$\text{def}(\dot A)=(n,n)$, $n\in \bbN\cup\{\infty\}$ is due to
Saakjan \cite{S1965}. 
Different approaches to
 the concept can be found in \cite{ABT}, 
\cite{DM},   \cite{GMT}, \cite{GT}, \cite{KL} and the  bibliography therein.
}.

Furthermore, Donoghue  showed that
the Weyl-Titchmarsh function admits the following Herglotz-Nevanlinna representation
\begin{equation}\label{hernev}
M(\dot A, A)(z)=\int_\bbR \left
(\frac{1}{\lambda-z}-\frac{\lambda}{1+\lambda^2}\right )
d\mu
\end{equation}
for some infinite Borel measure $\mu$ such that 
$$
\int_\bbR\frac{d\mu(\lambda)}{1+\lambda^2}=1,
$$
and discovered  the following result
(for a full presentation
 see
\cite{GT}).

\begin{theorem}[\cite{D}, \cite{GT}]\label{fund2}

Suppose that $\dot A$ and $\dot B$
 are  closed, prime, symmetric operators
with deficiency indices $(1,1)$. Assume, in addition, that
  $A$ and $B$ are some self-adjoint
extensions of $\dot A$ and $\dot B$, respectively.

Then
\begin{itemize}
\item[(i)] the pairs $(\dot A, A)$ and $(\dot B, B)$ are unitarily
equivalent\footnote{We say that pairs of operators
$(\dot A, A)$ and $(\dot B, B)$ in Hilbert spaces $\cH_A$ and $\cH_B$
are unitarily equivalent if there is
a unitary map $\cU$ from $\cH_A$ onto $\cH_B$ such that
$\dot B=\cU\dot A\cU^{-1}$ and  $ B=\cU A\cU^{-1}$.}
 if, and only if, the Weyl-Titchmarsh  functions $M(\dot A, A)$ and
$M(\dot B, B)$ coincide;

\item[(ii)] the pair $(\dot A, A)$ is unitarily equivalent to a
model pair $(\dot \cB, \cB)$ in the Hilbert space $L^2(\bbR;d\mu)$,
where
$\cB$ is the multiplication (self-adjoint)
operator by the  independent variable,
$\dot \cB$  is its
 restriction
on
\begin{equation}\label{nachalo}
\Dom(\dot \cB)=\left \{f\in \Dom(\cB)\, \bigg | \, \int_\bbR
f(\lambda)d \mu(\lambda) =0\right \},
\end{equation}
and $\mu$ is  the Borel  measure
from the Herglotz-Nevanlinna
representation \eqref{hernev}
for the Weyl-Titchmarsh function $M=M(\dot
A, A)$.

\end{itemize}
\end{theorem}

 Theorem \ref{fund2}, on the one hand, recognizes the
Weyl-Titchmarsh function $M$ as
a (complete) unitary invariant  of the pair of a
 symmetric operator with deficiency indices $(1,1)$ and its self-adjoint extension
that determines the operators uniquely up to unitary equivalence.
On the other hand, this result
 provides a general model
 for a symmetric operator with deficiency indices $(1,1)$ and its family of self-adjoint extensions.

The main goal of this paper is
\begin{itemize}
\item[(i)] to establish a relationship between
the classical analytic objects of  the extension
 theory such as the characteristic function $s(z)$ of a symmetric operator
$\dot A$,
the characteristic function $S(z)$ of
its  dissipative extensions $\widehat A$,
 and the
 Weyl-Titchmarsh function $M(z)$  associated with the pair $(\dot A, A)$;
\item[(ii)] to introduce the
 Weyl-Titchmarsh function for  dissipative  extensions
$\widehat A$ of $\dot A$,
and
\item[(iii)]  to
 provide a functional model
for the  triple $(\dot A, \widehat A, A)$ and obtain an analog of the formula of Krein 
\cite{K46} 
for the resolvents of $\widehat A$ and $A$.
\end{itemize}

The paper is organized as follows.

In Sec.~2 we propose to consider the characteristic function of a symmetric operator
$\dot A$  to be the one associated
with a pair $(\dot A, A)$ where $A$ is   a special reference
self-adjoint extension  
 of $\dot A$
uniquely determined by the choice of the basis $\{g_\pm\}$
in the deficiency subspace (see, eq. \eqref{star}).
 We call this function {\it the Liv\v{s}ic
function} associated with the pair $(\dot A, A)$. For 
different definitions of the  characteristic functions in the unbounded case
we refer to 
\cite{Ko80},
\cite{Ku}, 
\cite{P}, 
\cite{St68},
and 
\cite{TS}.

We also  show that the Weyl-Titchmarsh  and the Liv\v{s}ic
functions associated with
the pair $(\dot A , A)$ are related by the Cayley transformation
 (see Theorem \ref{fund22} (ii)).
Based on this transformation law, we show that  Liv\v{s}ic's
Theorem \ref{thmdiss} and Donoghue's Theorem \ref{fund2}
 can be deduced from one another.

In Sec.~3 we introduce the Weyl-Titchmarsh function of a dissipative
quasi-self-adjoint extension $\widehat A$ of $\dot A$
and extend Theorem \ref{fund22} to the dissipative
 case (see Theorem \ref{start}).

In Sec.~4 we introduce  the characteristic function
 associated with the
 triple of operators $(\dot A, \widehat A, A)$ and provide
a functional model for such  triples, continuing a list  of various
functional models for non-self-adjoint operators
discovered  by M. Liv\v{sic} \cite{L}, \cite{L1},
B. Sz.-Nagy and C. Foias \cite{NF}, L. de Branges
and J. Rovnyak \cite{dBR},
B.~S.~Pavlov \cite{P}, and others (also see
\cite{ABT},
\cite{Ku},
\cite{NV}, 
\cite{NK},
\cite{P2},
\cite{St68},
\cite{St98},  
 and references therein).

Based on  the functional model, we
obtain a refinement of Liv\v{s}ic's uniqueness result of
Theorem \ref{thmdiss} (see Theorem \ref{fund3}).

In Sec.~5 we perform initial spectral analysis of a dissipative triple
and obtain an analog of Krein's resolvent formula in the rank-one dissipative
extension  theory (cf.,  \cite{GMT}, \cite{GT}, \cite{Si}).

In Appendix A some spectral
 properties of the model symmetric operator are summarized.

Our exposition is based on the classical von Neumann extension theory. 
The proofs are straightforward and no knowledge
of the rigged Hilbert space approach 
\cite{ABT},
\cite{TS}
and/or  the boundary triplets theory 
 \cite{DM},
 \cite{MMog},
\cite{St68} is required. 

\section{The Liv\v{s}ic and the Weyl-Titchmarsh  functions }

 Liv\v{s}ic's definition  of a characteristic function
of a symmetric operator (see eq. \eqref{charsym})
has some ambiguity related to the
 choice of the deficiency elements $g_\pm$.
To avoid this ambiguity we proceed as follows.

Suppose that  $A$ is a self-adjoint extension of a symmetric operator
$\dot A$ with deficiency indices $(1,1)$. Let
$g_\pm$  be deficiency elements $g_\pm\in \Ker ((\dot A)^*\mp iI)$,
$\|g_+\|=1$. Assume, in addition, that
\begin{equation}\label{star}
g_+-g_-\in \Dom ( A).
\end{equation}

Introduce   the
Liv\v{s}ic function $s(\dot A, A)$
 associated with the pair $(\dot A, A)$ by
\begin{equation}\label{charf12}
s(z)=\frac{z-i}{z+i}\cdot \frac{(g_z, g_-)}{(g_z, g_+)}, \quad
z\in \bbC_+,
\end{equation}
where $0\ne g_z\in \Ker((\dot A)^*-zI )$ is   an arbitrary
(deficiency) element.

 The following result establishes a standard relationship between the
 Weyl-Titchmarsh and the    Liv\v{s}ic functions associated with the  pair
$(\dot A, A) $  \cite{ABT}, \cite{Br},
 \cite{BL58}, \cite{TS} 
(also see \cite{DM}, where a linear--fractional transformation between the
 Liv\v{s}ic characteristic function of a symmetric operator 
and the  Krein--Langer $Q$-function in the framework of boundary 
triplets theory was established).

\begin{theorem}\label{drob}
Denote by $M=M(\dot A, A)$ and by $s=s(\dot A, A)$
  the Weyl-Titchmarsh function and the    Liv\v{s}ic function  associated
with the pair $(\dot A, A)$, respectively.

 Then
\begin{equation}\label{blog}
s(z)=\frac{M(z)-i}{M(z)+i},\quad z\in \bbC_+.
\end{equation}

\end{theorem}
\begin{proof}
Let $\mu$ be the measure from the Herglotz-Nevanlinna representation \eqref{hernev}
of the
Weyl-Titchmarsh function $M$ associated with the pair
 $(\dot A,A)$.

By splitting  off the self-adjoint part of $\dot A$, we may assume
that $\dot A$ is a prime symmetric operator and then, in
accordance with  Theorem \ref{fund2}, one may also assume without
loss that $\dot A$ and $A$ are already chosen in their model
representation in the Hilbert space $L^2(\bbR; d\mu)$, so that the deficiency elements 
are given by
$$
g_z(\lambda)=\frac{1}{\lambda-z},\quad
\,\, \text{$\mu$-a.e. }, \quad \Im(z)\ne 0,
$$
and
$$
g_\pm(\lambda)=g_{\pm i}(\lambda)=\frac{1}{\lambda\mp i}, \quad
\,\, \text{$\mu$-a.e. }.
$$
Notice that $g_+-g_-\in \Dom (A) $ and hence
$$
s(z)=\frac{z-i}{z+i}\cdot \frac{(g_z, g_-)}{(g_z, g_+)}
=\frac{(z-i)\int_\R \frac{d\mu(\lambda)}{(\lambda-z)(\lambda-i)}}
{(z+i)\int_\R \frac{d\mu(\lambda)}{(\lambda-z)(\lambda+i)}}
=\frac{M(z)-i}{M(z)+i},
$$
completing the proof.
\end{proof}

As a corollary we obtain the following analog of Theorem
\ref{fund2}.

\begin{theorem}\label{fund22}

Suppose that $\dot A$ and $\dot B$
 are  closed, prime, densely defined  symmetric operators
with deficiency indices $(1,1)$. Assume, in addition, that
  $A$ and $B$ are some self-adjoint
extensions of $\dot A$ and $\dot B$ respectively.

Then

\begin{itemize}
\item[(i)] the pairs $(\dot A, A)$ and $(\dot B, B)$ are unitarily
equivalent if, and only if, the Liv\v{s}ic functions $s(\dot A, A)$ and
$s(\dot B, B)$ coincide;

\item[(ii)] the pair $(\dot A, A)$ is unitarily equivalent to the
model pair $(\dot \cB, \cB)$ in the Hilbert space $L^2(\bbR;d\mu)$,
where
 $\mu $ is  the representing measure  for the Weyl-Titchmarsh function
$M=M(\dot A, A)$.

In this case,
$$
M(z)=\frac1i\cdot\frac{s(z)+1}{s(z)-1}, \quad z\in \bbC_+,\quad
\text
{with} \quad s=s(\dot A, A).$$
\end{itemize}
\end{theorem}

\section{The Weyl-Titchmarsh function in the dissipative case}

The Weyl-Titchmarsh function can also be introduced
 in a more general context of
dissipative extensions of a symmetric operator. We refer to \cite{AHS}
 (also see \cite{ABT})
where the concept of the Weyl-Titchmarsh function for
 bounded non-self-adjoint operators has been discussed.

To be more specific, suppose that $\dot A$ is a  densely
defined, closed, symmetric operator with deficiency indices $(1,1)$ and
$\widehat A$ is  its  maximal dissipative
extension, that is,
$$
\dot A \subset\widehat A\subset (\dot A)^*$$ and
$$
\Im(\widehat Af,f)\ge 0, \quad f\in \Dom(\widehat A).
$$

Suppose that $g_+$ is a deficiency element $g_+\in\Ker ((\dot
A)^*-iI)$
 satisfying the normalization condition $||g_+||=1$.

We introduce the Weyl-Titchmarsh function $\cM=\cM(\dot A, \widehat A)$ associated with the
pair $(\dot A, \widehat A)$ by
$$
\cM(z)=\left (((\widehat A)^*z+I)((\widehat A)^*-zI)^{-1}g_+,g_+\right ), \quad
z\in \bbC_+.
$$

We remark that in contradistinction to the self-adjoint case, the
dependence of $\cM$ on the first argument, the symmetric
operator $\dot A$, can be suppressed for, under our
assumptions,
$$
\dot A=\widehat A|_{\dom (A)\,\cap \,\dom ((\dot A)^*)}.
$$
In other words,  the symmetric operator $\dot A$ is uniquely determined
by the  extension $\widehat A$, provided that $\widehat A$
 is not self-adjoint.

In order to establish a relationship between the Liv\v{s}ic function associated
with the pair $(\dot A, A) $ and
 the Weyl-Titchmarsh function of the dissipative extension $\widehat A$
 of
$\dot A$  it is convenient to use an analog of
the von Neumann parameterization of $\Dom (\widehat A)$.

To set up the notation, introduce the following hypothesis.

\begin{hypothesis}\label{setup} Suppose 
that $\widehat A\ne(\widehat A)^*$  is  a maximal 
dissipative extension of  a symmetric operator $\dot A$
 with deficiency indices $(1,1)$. Assume, in addition, that
$A$ is a  self-adjoint extension of $\dot A$. Suppose,  that the
deficiency elements $g_\pm\in \Ker ((\dot A)^*\mp iI)$ are
normalized, $\|g_\pm\|=1$, and chosen in such a way that
\begin{equation}\label{ddoomm14}g_+- g_-\in \dom ( A)\,\,\,\text{and}\,\,\,
g_+-\kappa g_-\in \dom (\widehat A)\,\,\,\text{for some }
\,\,\,|\kappa|<1.
\end{equation}
\end{hypothesis}

The following result is an extension  of Theorem \ref{drob} to the
dissipative case.

%%%%%%%%%%%%%%%%%%%%%%%%%%%%%%%%%%%%
%%%%%%%%%%%%%%%%%%%%%%%%%%%%%%%%%%%%%
%%%%%%%%%%%%%%%%%%%%%%%%%%%%%%%%%%%%%%%%
\begin{theorem}\label{start}
 Assume Hypothesis \ref{setup}. Denote by  $s=s(\dot A, A)$  the Liv\v{s}ic function  associated with the
pair $(\dot A, A)$
  and by $\cM=\cM( \widehat A)$
 the Weyl--Titchmarsh function of the dissipative operator $\widehat A$.

Then
\begin{equation}\label{sviaM}
\overline{\kappa}s(z) =\frac{\cM(z)-i}{\cM(z)+i},\quad z\in \bbC_+.
\end{equation}

\end{theorem}

%%%%%%%%%%%%%%%%%%%%%%%%%%%%%%%%%%%%%%%%%%%
%%%%%%%%%%%%%%  PROOF %%%%%%%%%%%%%%%%%%%%%%%%
%%%%%%%%%%%%%%%%%%%%%%%%%%%%%%%%%%%%%%%%%%%

\begin{proof} Since the
 extension $\widehat A$ of $\dot A$
is a restriction of $(\dot A)^*$ on $\Dom (\widehat A)$, one gets
that
\begin{equation}\label{sososo}
(\widehat A -iI)(\widehat A-zI)^{-1}g_+\in \Ker ((\dot A)^*-zI),
\quad z\in  \rho (\widehat A),
\end{equation}
where $\rho(\widehat A)$ denotes the resolvent set of $\widehat A$.

Indeed,
$$
(\widehat A -iI)(\widehat A-zI)^{-1}g_+=g_++(z-i)(\widehat
A-zI)^{-1}g_+
$$
and hence
\begin{align*}
((\dot A)^*-zI)&(\widehat A -iI)(\widehat A-zI)^{-1}g_+ = ((\dot
A)^*-zI)g_+\\&+(z-i)((\dot A)^*-zI)(\widehat A-zI)^{-1}g_+
\\ &=
(i-z)g_++(z-i)(\widehat A-zI)(\widehat A-zI)^{-1}g_+=0
\end{align*}
which proves \eqref{sososo}.

From \eqref{sososo} one obtains that
 $(\widehat A -iI)(\widehat A +iI)^{-1} g_+\in \Ker
((\dot A)^*+iI)$ and hence
 $
(\widehat A -iI)(\widehat A+iI)^{-1} g_+=\alpha g_- $
for some $ \alpha \in \bbC.
$

On the other hand,
\begin{equation}\label{sasa}
\alpha g_-=(\widehat A -iI)(\widehat A+iI)^{-1}
g_+=g_+-2i(\widehat A+iI)^{-1}g_+\
\end{equation}
and therefore
$$
\alpha g_--g_+=-2i(\widehat A+iI)^{-1}g_+\in \Dom (\widehat A).
$$
Taking into account the characterization \eqref{ddoomm14}
 of $ \Dom(\widehat A)$, one obtains that
$\alpha =\kappa$ and therefore
\begin{equation}\label{didi}
\kappa g_-=(\widehat A -iI)(\widehat A+iI)^{-1} g_+
,\end{equation} as it follows from \eqref{sasa}.

Introducing the elements
\begin{equation}\label{dididi}
g_z=(\widehat A^*-iI)(\widehat A^*-zI)^{-1}g_+, \quad z\in
\rho(\widehat A^*),
\end{equation}
and taking into account that the adjoint operator $(\widehat  A)^*$ is also
 a quasi-self-adjoint extension of $\dot A$, one concludes that
$$
g_z\in \Ker ((\dot A)^*-zI), \quad z\in \rho (\widehat  A),
$$
which shows that the Liv\v{s}ic function $s=s(\dot A , A)$ admits the representation
(see definition \eqref{charf12})
$$
s(z)=\frac{z-i}{z+i}\cdot \frac{(g_z,g_-)}{(g_z,g_+)}, \quad z\in
\bbC_+.
$$
Therefore, in view of \eqref{didi} and \eqref{dididi}, one computes
\begin{align*}
\overline{\kappa}\,
s(z)&=\frac{z-i}{z+i}\cdot\frac{(g_z,\kappa g_-)}
{(g_z,g_+)}\\
&= \frac{z-i}{z+i}\cdot \frac{((\widehat A^*-iI)(\widehat
A^*-zI)^{-1}g_+, (\widehat A -iI)(\widehat A+iI)^{-1}g_+)} {(
(\widehat A^*-iI)(\widehat A^*-zI)^{-1}g_+         , g_+)}
\\&=
 \frac{z-i}{z+i}\cdot \frac{((\widehat A^*+iI)(\widehat A^*-zI)^{-1}g_+,
g_+)} {( (\widehat A^*-iI)(\widehat A^*-zI)^{-1}g_+         ,
g_+)}
\\&=
\frac{z-i}{z+i}\cdot \frac{(1+(z+i)((\widehat A^*-zI)^{-1}g_+,
g_+))} {(1+(z-i)((\widehat A^*-zI)^{-1}g_+, g_+))}
\\&=
  \frac{z-i+(z^2+1)((\widehat A^*-zI)^{-1}g_+,
g_+)} {z+i+(z^2+1)((\widehat A^*-zI)^{-1}g_+, g_+)}
\\&=
 \frac{\cM(z)-i}{\cM(z)+i}\,,\quad z\in
\rho(\widehat A^*),
\end{align*}
proving the claim.

\end{proof}
\begin{remark} Combining Theorems \ref{drob} and  \ref{start}, it is easy
to see that under Hypothesis \ref{setup} the Weyl--Titchmarsh functions
$M=M(\dot A,  A)$ and $\cM(\widehat A)$ are related as follows
$$
\frac{\cM-i}{\cM+i}=\overline{\kappa} \cdot \frac{M-i}{M+i}.
$$
\end{remark}
Our next result shows that the
Weyl--Titchmarsh function of a dissipative operator admits the Herglotz-Nevanlinna
representation
with a (locally)   absolutely continuous representing measure.

\begin{corollary}\label{realiza} Assume Hypothesis \ref{setup}.
Then    the Weyl--Titchmarsh function $\cM=\cM( \widehat A)$   of
the dissipative operator $\widehat A$ is an analytic function
mapping the upper half-plane into the disk of radius
$\frac{2|\kappa|}{1-|\kappa|^2}$ centered at the point
$\left (0,\frac{1+|\kappa|^2}{1-|\kappa|^2} \right )$ of the
$xy$-plane.

Moreover, the function  $\cM$ admits the Herglotz-Nevanlinna representation
\begin{equation}\label{vseaaa}
\cM(z)=\int_\bbR \left
(\frac{1}{\lambda-z}-\frac{\lambda}{1+\lambda^2}\right )
f(\lambda)d\lambda, \quad z\in \bbC_+,
\end{equation}
where
\begin{equation}\label{rnest}
\frac1\pi \cdot \frac{1-|\kappa|}{1+|\kappa|}\le f(\lambda)
\le\frac1\pi \cdot\frac{1+|\kappa|}{1-|\kappa|} \quad \text{a.e.}\,
\end{equation}
and
\begin{equation}\label{vsea1aa}
\int_\R \frac{f(\lambda)}{1+\lambda^2}d\lambda=1.
\end{equation}

\end{corollary}
\begin{proof}
Since the Cayley transform
$$z\mapsto \frac{z-i}{z+i}$$ maps the disk
$\frac{2|\kappa|}{1-|\kappa|^2}\bbD+i\frac{1+|\kappa|^2}{1-|\kappa|^2}$
onto the disk $|\kappa|\bbD$ and the Liv\v{s}ic function
$s(\dot A, A)$ is contractive in $\bbC_+$,
  from \eqref{sviaM} follows that
$$
\text{Range }( \cM)\subset \frac{2|\kappa|}
{1-|\kappa|^2}\bbD+i\frac{1+|\kappa|^2}{1-|\kappa|^2}.
$$
In particular,
\begin{equation}\label{rnestt}
\frac{1-|\kappa|}{1+|\kappa|}\le\Im \,\cM(z)\le
\frac{1+|\kappa|}{1-|\kappa|} , \quad z\in \bbC_+.
\end{equation}

Since $\cM(i)=i$ and $\cM$ has a bounded imaginary part in the upper
half-plane, using Fatou's Lemma, one proves
 the Herglotz-Nevanlinna representation
$$
\cM(z)=\int_\bbR \left
(\frac{1}{\lambda-z}-\frac{\lambda}{1+\lambda^2}\right )
d\mu(\lambda), \quad z\in \bbC_+,
$$
where $\mu$ is a (locally) absolutely
 continuous measure with the Radon-Nikodym density
$$
f(\lambda)=\frac1\pi\lim_{\varepsilon \downarrow 0}\Im \,\cM(\lambda+i\varepsilon),
 \quad \text{a.e.}\,,
$$
which proves  \eqref{rnest}, using  \eqref{rnestt}. Now 
  \eqref{vsea1aa} follows from the observation that $\cM(i)=i$.
\end{proof}

\section{The characteristic function and the functional model of a dissipative operator}

Under Hypothesis \ref{setup}, we  introduce  the characteristic
function $S=S( \dot A, \widehat A, A)$
associated with the triple of operators $( \dot A, \widehat A, A)$
as the M\"obius transformation
\begin{equation}\label{ch1}
S(z)=\frac{s(z)-\kappa} {\overline{ \kappa }\,s(z)-1}, \quad z\in \bbC_+,
\end{equation}
of the Liv\v{s}ic function $s=s(\dot A, A)$ associated with the pair
$(\dot A, A)$.

We remark that given  a triple  $( \dot A, \widehat A, A)$, one can always find a
basis $g_\pm$ in the deficiency subspace
 $\Ker ((\dot A)^*-iI)\dot +\Ker ((\dot A)^*+iI)$,
$$\|g_\pm\|=1, \quad g_\pm\in ((\dot A)^*\mp iI),
$$ such that
$$
g_+-g_-\in \Dom (A)
\quad \text{
and}\quad
g_+-\kappa g_-\in \Dom (\widehat A),
$$
and then, in this case,
\begin{equation}\label{assa}
\kappa =S( \dot A, \widehat A, A)(i).
\end{equation}

Our next goal is to introduce a {\it functional model} of a prime
 dissipative triple\footnote{We call a triple $(\dot A, \widehat A, A)$
 a prime triple if $\dot A$ is a prime symmetric operator.} parameterized by the characteristic
function.

Given a contractive analytic map $S$,
\begin{equation}\label{chchch}
S(z)=\frac{s(z)-\kappa} {\overline{ \kappa }\,s(z)-1}, \quad z\in \bbC_+,
\end{equation}
where $|\kappa|<1$ and $s$ is 
an  analytic, contractive function in $\bbC_+$ 
satisfying the Liv\v{s}ic criterion \eqref{vsea0}, introduce the function
$$
M(z)=\frac1i\cdot\frac{s(z)+1}{s(z)-1},\quad z\in \bbC_+,
$$
so that
$$
M(z)=\int_\bbR \left
(\frac{1}{\lambda-z}-\frac{\lambda}{1+\lambda^2}\right )
d\mu(\lambda), \quad z\in \bbC_+,
$$
for some infinite Borel measure with
$$
\int_\bbR\frac{d\mu(\lambda)}{1+\lambda^2}=1.
$$

In the Hilbert space $L^2(\bbR;d\mu)$ introduce
  the multiplication (self-adjoint)
operator by the  independent variable $\cB$
 on
\begin{equation}\label{nacha1}
\Dom(\cB)=\left \{f\in \,L^2(\bbR;d\mu) \,\bigg | \, \int_\bbR
\lambda^2 | f(\lambda)|^2d \mu(\lambda)<\infty \right \},
\end{equation} denote by  $\dot \cB$  its
 restriction
on
\begin{equation}\label{nacha2}
\Dom(\dot \cB)=\left \{f\in \Dom(\cB)\, \bigg | \, \int_\bbR
f(\lambda)d \mu(\lambda) =0\right \},
\end{equation}
and let
 $\widehat \cB$ be   the dissipative restriction of the operator  $(\dot \cB)^*$
 on
\begin{equation}\label{nacha3}
\Dom(\widehat \cB)=\dom (\dot \cB)\dot +\linspan\left
\{\,\frac{1}{\cdot -i}- S(i)\frac{1}{\cdot +i}\right \}.
\end{equation}

We will refer to the triple  $(\dot \cB,   \widehat \cB,\cB)$ as
{\it  the model
 triple } in the Hilbert space $L^2(\bbR;d\mu)$.

Our next result shows that a triple $(\dot A,A,\widehat A)$
with the characteristic function $S$
is unitarily equivalent to the model triple
$(\dot \cB,   \widehat \cB,\cB)$
in the Hilbert space $L^2(\bbR;d\mu)$
whenever the underlying symmetric operator $\dot A$ is prime.

The triple $(\dot \cB,   \widehat \cB,\cB)$ will therefore be called {\it the functional model}
for $(\dot A,A,\widehat A)$.

\begin{theorem}\label{fund3}

Suppose that $\dot A$ and $\dot B$
 are prime, closed, densely defined  symmetric operators
with deficiency indices $(1,1)$. Assume, in addition, that
  $A$ and $B$ are some self-adjoint
extensions of $\dot A$ and $\dot B$ and that
$\widehat A$ and $\widehat B$ are maximal dissipative
extensions of $\dot A$ and $\dot B$, respectively
($\widehat A\ne (\widehat A)^*$, 
$\widehat B\ne (\widehat B)^*$).

Then
\begin{itemize}
\item[(i)] the triples  $(\dot A,\widehat A, A)$ and $(\dot
B,\widehat B,  B)$
 are unitarily equivalent\footnote{We say that triples  of operators
$(\dot A, \widehat A, A)$ and $(\dot B,\widehat B,  B)$
 in Hilbert spaces $\cH_A$ and $\cH_B$
are unitarily equivalent if there is
a unitary map $\cU$ from $\cH_A$ onto $\cH_B$ such that
$\dot B=\cU\dot A\cU^{-1}$, $\widehat  B=\cU\widehat  A\cU^{-1}$, and  $ B=\cU A\cU^{-1}$.}
if, and only if, the characteristic functions   $S_A=S(\dot A,\widehat  A, A)$
and $S_B=S(\dot B, \widehat B, B)$ of the triples coincide;

\item[(ii)] the triple  $(\dot A,\widehat A,  A)$ is unitarily
equivalent to the model triple $(\dot \cB, \widehat\cB,  \cB )$ in
the Hilbert space $L^2(\bbR;d\mu)$, where $\mu$ is the
representing measure for the Weyl--Titchmarsh function $M=M(\dot A,
A)$ associated with the pair $(\dot A, A)$.

\end{itemize}

\end{theorem}

\begin{proof} (i).
 Since
$
S_A=S_B
$
and $$
s( \dot A, A)=\frac{S_A-S_A(i)}{\overline{S_A(i)}
S_A-1}=\frac{S_B-S_B(i)}{\overline{S_B(i)} S_B-1}=s( \dot B, B),
$$
one concludes that the Liv\v{s}ic functions $s( \dot A, A)$  and $
s(\dot B, B)$ coincide,

By Theorem \ref{fund2} (i), there exists  a unitary map $\cU$ such
that
\begin{equation}\label{use}
\cU \dot A\cU^*=\dot B \quad \text{ and }\quad \cU A\cU^*= B.
\end{equation}

Let $g_\pm\in \Ker((\dot A)^*\mp i I)$, $\|g_\pm\|=1$, such that
$$
g_+-g_-\in\Dom(A).
$$
Set
$$
f_\pm=\cU g_\pm.
$$
Equalities \eqref{use} yield $f_\pm\in \Ker((\dot B)^*\mp i I)$,
$\|f_\pm\|=1$, and
$$
f_+-f_-\in\Dom(B).
$$
In this case, since $
\kappa=S_A(i)=S_B(i)
$,
the membership $$g_+-\kappa g_-\in \Dom(\widehat
A)$$ means that  $f_+-\kappa f_-\in \Dom(\widehat
B)$. Thus,
$$
\cU\Dom(\widehat A)=\Dom(\widehat B),
$$
and hence
$$\cU\,\widehat A\cU^*=\widehat B,$$ proving that the triples
$(\dot A, \widehat A, A)$ and $(\dot B, \widehat B, B)$ are
unitarily equivalent.

(ii). By Theorem \ref{fund2} (ii), the Weyl--Titchmarsh function
$M(\dot \cB, \cB)$ coincides with
 $M(\dot A, A)$ and hence the corresponding Liv\v{s}ic functions coincide, that is,
$
s(\dot A, A)=s(\dot \cB, \cB).
$

Since the deficiency elements
$$
f_\pm(\lambda)=\frac{1}{\lambda\mp i}
$$ are normalized by one,
$
f_+-f_-\in \Dom (\cB)
$,
and, by hypothesis,
$$
f_+-S(i)f_-\in \Dom (\widehat \cB)
,$$
one computes that
$$
S_B=\frac{s(\dot \cB, \cB)-S_A(i)}{\overline{S_A(i)}s(\dot \cB, \cB)-1}
=S_A.$$ Therefore, the triples $(\dot A, \widehat
A, A)$ and $(\dot \cB, \widehat \cB, \cB)$ are unitarily
equivalent by the first part of the proof.

The proof is complete.
\end{proof}

\begin{remark} We remark that
the Weyl-Titchmarsh function
$M=M(\dot A, A)$
 can be recovered from the characteristic function $S_A$ of the triple
$(\dot A, \widehat A, A)$ by  the  equation
$$
\quad M(z)=\frac1i\cdot \frac{s(z)+1}{s(z)-1}\quad \text{with}
\quad
s(z)=\frac{S_A(z)-S_A(i)}{\overline{S_A(i)} S_A(z)-1}, \quad z\in \bbC_+.
$$

\end{remark}

We conclude this section by showing  that the characteristic
function $S$ associated with the triple
$(\dot A,\widehat A, A)$ and the Weyl-Titchmarsh function $\cM(
\widehat A)$ of the dissipative operator $\widehat A$
  are related by  a linear transformation.

%%%%%%%%%%%%%%%%%%%%%%%%%%%%%%%%%%%%%%%%%

\begin{theorem}\label{kolo}
Let
 $\cM$ be the Weyl-Titchmarsh function $\cM( \widehat A)$ of a
 dissipative operator $\widehat A$
and  $S=S(\dot A,\widehat A, A)$  the characteristic function  associated
with the triple  $(\dot A, \widehat A, A)$.

Then
 \begin{equation}\label{kilo}
\overline{ \kappa }S(z)=
\frac{|\kappa|^2-1}{2i}\cM(z)+\frac{|\kappa|^2+1}{2},
\text{ with }
\kappa=S(i), \quad z\in \bbC_+.\end{equation}
\end{theorem}

\begin{proof}
By definition,  the  characteristic function
$S$ and the Liv\v{s}ic function  $s$ are related by the
M\"obius transformation
$$S=
\frac{s-\kappa}{\overline{\kappa}s-1}.
$$
By Theorem \ref{start}, one obtains that
$$
\overline{ \kappa }s=\frac{\cM-i}{\cM+i}
$$and therefore
\begin{align*}
\overline{\kappa}S&=\overline{\kappa}\cdot
\frac{s-\kappa}{\overline{\kappa}s-1} =
\frac{\overline{\kappa}s-|\kappa|^2}{\overline{\kappa}s-1}=
\frac{\frac{\cM-i}{\cM+i}-|\kappa|^2}{\frac{\cM-i}{\cM+i}-1}\\&=\frac{i}{2}
\left ((1-|\kappa|^2)\cM-i(1+|\kappa|^2)\right )
\end{align*}
which proves \eqref{kilo}.
\end{proof}

\begin{remark}

Theorem \ref{kolo} shows that in case when the von Neumann
 parameter $\kappa$ does not vanish,
 the characteristic function $S$ of a (prime)
triple
is  uniquely determined
by
  the pair $
(\kappa, \cM)$. Therefore, along with the characteristic function $S$,
the pair $
(\kappa, \cM)$, with $\kappa\ne 0$,
can also be considered
to be  a complete unitary  invariant of a prime dissipative  tripple
 $(\dot A, \widehat A, A)$.

However, if $\kappa=0$, and therefore by \eqref{sviaM},
$\cM(z)=i$ for all $z\in \bbC_+$,
the quasi-self-adjoint extension $\widehat A$
coincides with the restriction 
of the adjoint operator $(\dot A)^*$ on
$$
\Dom(\widehat A)=\Dom(\dot A)\dot + \Ker ((\dot A)^*-iI).
$$
Hence, the  prime triples $(\dot A, \widehat A, A)$ with
$\kappa=0$
 are  in one-to-one correspondence with the set of prime symmetric operators.

In this case, the characteristic function $S$
and the Liv\v{s}ic function $s$ coincide (up to a sign),
$$
S(z)=-s(z), \quad z\in \bbC_+.
$$
Therefore, for $S$ to be  the characteristic function
of a triple  $(\dot A, \widehat A, A)$ with
$\kappa =0$ and $\cM(z)=i$ for all $z$  in the upper half-plane it is necessary and sufficient that
$S$ satisfy the Liv\v{s}ic criterion \eqref{vsea0}.

\end{remark}

\section{The spectral analysis of the model  dissipative operator}

In the suggested functional model in the Hilbert space $L^2(\bbR;d\mu)$,
the eigenfunctions of the model dissipative operator
 $\widehat \cB$ from the triple $(\dot \cB, \widehat \cB,
 \cB)$
look exceptionally simple.

\begin{lemma}\label{specpoint} Suppose that
$(\dot \cB, \widehat \cB, \cB)$ is the model triple  in $L^2(\bbR;d\mu)$.
 Then a point $z_0\in \bbC_+$ is an eigenvalue
of the dissipative operator  $\widehat \cB$ if,
and only if, $S(\dot \cB, \widehat \cB, \cB)(z_0)=0$.

 In
this case, the corresponding eigenfunction $f$ is of the form
$$
f(\lambda)=\frac{1}{\lambda-z_0},\quad
\,\, \text{$\mu$-a.e. }
$$
\end{lemma}
\begin{proof} Suppose that $z_0\in \bbC_+$ is an eigenvalue
of $\widehat \cB$ and
that $f$, $f\in L^2(\bbR;d\mu)$, is the corresponding eigenvector,
that is,
$$
\widehat \cB f=z_0 f, \quad f\in\Dom( \widehat \cB ).$$ Since
$f\in\Dom( \widehat \cB)$, the element $f$ admits the
representation
$$
f(\lambda)=f_0(\lambda)+K\left (\frac{1}{\lambda-i}-\kappa
\frac{1}{\lambda+i}\right ),
$$
where $f_0\in \Dom (\dot \cB)$ and $K$ is some constant. Then
$$
0=((\widehat \cB-z_0I)f)(\lambda)=(\lambda-z_0)f_0(\lambda)+K\left
(\frac{i-z_0}{\lambda-i}+\kappa \frac{i+z_0}{\lambda+i}\right )
$$
and hence
$$
f_0(\lambda)=-\frac{K}{\lambda-z_0}\left
(\frac{i-z_0}{\lambda-i}+\kappa \frac{i+z_0}{\lambda+i}\right ) .$$
Since $f_0\in \Dom (\dot \cB)$,
$$
\int_\bbR f_0(\lambda) d\mu(\lambda)=0
$$
and hence
\begin{align*}
0&=\int_\bbR\frac{1}{\lambda-z_0}\left
(\frac{i-z_0}{\lambda-i}+\kappa \frac{i+z_0}{\lambda+i}\right )
d\mu(\lambda)
\\&=-\int_\bbR\left (\frac{1}{\lambda-z_0}-\frac{1}{\lambda-i}\right )
d\mu(\lambda)
+\kappa \int_\bbR\left
(\frac{1}{\lambda-z_0}-\frac{1}{\lambda+i}\right )d\mu(\lambda)
\\&=
-M(z_0)+M(i)+\kappa( M(z_0)-M(-i))
\\&=-M(z_0)+i+\kappa( M(z_0)+i)
.\end{align*} Therefore,
$$
\kappa =\frac{M(z_0)-i}{M(z_0)+i}=s(\dot \cB, \widehat \cB)(z_0)
$$
and hence the characteristic  function $S(\dot \cB,\widehat \cB, \cB)$
vanishes at the point $z_0$,
$$
S(\dot \cB, \widehat \cB, \cB)(z_0)=\frac{s(\dot \cB, \widehat
\cB)(z_0)-\kappa} {\overline{\kappa}s(\dot \cB, \widehat
\cB)(z_0)-1}=0.
$$
In this case,
\begin{align*}
f(\lambda)&=K\left [\left (\frac{1}{\lambda-i}-\kappa
\frac{1}{\lambda+i}\right ) -\frac{1}{\lambda-z_0}
\left (\frac{i-z_0}{\lambda-i}+\kappa \frac{i+z_0}{\lambda+i}\right )
\right ]\\& =K\left [\frac{1}{\lambda-i}\left  (1-\frac{i-z_0}{\lambda-z_0}\right )
-\kappa\frac{1}{\lambda+i}\left (1+\frac{i+z_0}{\lambda-z_0}\right ) \right ]
\\&
=K\frac{1-\kappa}{\lambda-z_0}.
\end{align*}
So, we have shown that if $ z_0$ is an eigenvalue of
of $\widehat \cB$,  then
$$
S(\dot \cB,\widehat \cB, \cB)(z_0)=0
$$
and that the corresponding eigenelement $f$ is of the form
\begin{equation}\label{eigenf}
f(\lambda)=\frac{1}{\lambda-z_0}.
\end{equation}

 Repeating the same reasoning in the reverse order, one easily shows that
if $S(\dot \cB, \widehat \cB, \cB)(z_0)=0$, then the function
$f$ given by \eqref{eigenf}
belongs to $\Dom (\widehat \cB)$ and $\widehat \cB f=z_0f $.
\end{proof}

For the resolvents of the model dissipative operator $\widehat\cB$
and the self-adjoint (reference)
operator $\cB$ from the model  triple $(\dot \cB, \widehat \cB, \cB)$
 one gets
 the following resolvent formula.
\begin{theorem} Suppose that
$(\dot \cB, \widehat \cB, \cB)$ is the model triple in the Hilbert space
$L^2(\bbR;d\mu) $.

 Then the
resolvent
 of the model dissipative operator $\widehat \cB$  in
$L^2(\bbR;d\mu) $ has the form
$$
(\widehat \cB- zI )^{-1}=(\cB- zI )^{-1}-p(z)(\cdot\, ,
g_{\overline{z}})g_z , $$ with
$$
p(z)=\left (M(\dot \cB,
\cB)(z)+i\frac{\kappa+1}{\kappa-1}\right )^{-1},
$$
$$z\in\rho(\widehat \cB)\cap \rho(\cB).
$$
Here $M(\dot \cB,
\cB)$ is the Weyl-Titchmarsh function associated with the pair
 $(\dot \cB, \cB)$ continued to the lower half-plane by the Schwarz reflection
principle,
and the deficiency elements $g_z$ are given by
$$
g_z(\lambda)=\frac{1}{\lambda-z}, \quad
\,\, \text{$\mu$-a.e. }.
$$

\end{theorem}

\begin{proof} Given $h\in L^2(\bbR; d\mu)$ and $z\in \rho(\widehat \cB)$,
 suppose that
\begin{equation}\label{jkl}
(\widehat \cB- zI)f=h \quad \text{ for  some } f\in \Dom(\widehat
\cB).
\end{equation}
Since $f\in \Dom(\widehat \cB)$, one gets the representation
\begin{equation}\label{eqff}
f(\lambda)=f_0(\lambda)+K\left (\frac{1}{\lambda-i}-
\kappa\frac{1}{\lambda+i}\right )
\end{equation}
for some $f_0\in \Dom(\dot \cB)$ and $K\in \bbC$. Eq. \eqref{jkl}
yields
$$
(\lambda-z)f_0(\lambda)+K\left (\frac{i-z}{\lambda-i}+
\kappa\frac{i+z}{\lambda+i}\right )=h(\lambda)
$$
and hence
\begin{equation}\label{eqff1}
f_0(\lambda)=\frac{h(\lambda)}{\lambda-z}-
\frac{K}{\lambda-z}\left (\frac{i-z}{\lambda-i}+
\kappa\frac{i+z}{\lambda+i}\right ).
\end{equation}
Since $f_0\in \Dom(\dot A)$, and therefore
$$
\int_\bbR f_0(\lambda)d\mu(\lambda)=0,
$$
integrating \eqref{eqff1} against $\mu$, one obtains that
\begin{equation}\label{KK}
K \int_\bbR \frac{1}{\lambda-z}\left (\frac{i-z}{\lambda-i}+
\kappa\frac{i+z}{\lambda+i}\right )d\mu(\lambda)= \int_\bbR
\frac{h(\lambda)}{\lambda-z}d\mu(\lambda).
\end{equation}
Observing that
$$
\int_\bbR \frac{1}{\lambda-z}\left (\frac{i-z}{\lambda-i}+
\kappa\frac{i+z}{\lambda+i}\right
)d\mu(\lambda)=i-M(z)+\kappa(M(z)+i),
$$ with
$
M(z)=M(\dot \cB, \widehat \cB)(z)
$, and solving \eqref{KK} for $K$, one obtains
$$
K=\frac{\int_\bbR
\frac{h(\lambda)}{\lambda-z}d\mu(\lambda)}{(\kappa-1)M(z)+i(1+\kappa)}
.$$ Combining \eqref{eqff} and \eqref{eqff1}, for the element $f$
we have the representation
\begin{align*}
f(\lambda)&=\frac{h(\lambda)}{\lambda-z}+ \frac{K}{\lambda-z}\left
(\frac{\lambda-z}{\lambda-i}-
\kappa\frac{\lambda-z}{\lambda+i}- \left
[\frac{i-z}{\lambda-i}+ \kappa\frac{i+z}{\lambda+i}\right
)\right ]
\\&=
\frac{h(\lambda)}{\lambda-z} -K\frac{\kappa-1}{\lambda-z}
\\&
=\frac{h(\lambda)}{\lambda-z}-\left ( M(z) +
i\frac{\kappa+1}{\kappa-1}
\right)^{-1}\frac{1}{\lambda-z}\int_\bbR
\frac{h(\lambda)}{\lambda-z}d\mu(\lambda),
\\&\quad \quad \quad \quad \quad z\in \rho(\widehat \cB)\cap\rho(\cB),
\end{align*}
which proves the claim.
\end{proof}

\begin{remark}  It is
 easy to see, using \eqref{blog}, that the poles of the function
$p$ in the upper half-plane coincide with the roots of the equation
$$
s(\dot \cB, \cB)(z)=\kappa, \quad z\in \bbC_+,
$$
provided that $\kappa \ne 0$ and $M(z)\ne i$  identically
in the upper half-plane.
Therefore, the zeros of the characteristic function $S(\dot \cB, \widehat \cB, \cB)$
in the upper half-pane determine the poles of the resolvent of
 the dissipative operator
$\widehat \cB$ (cf., Lemma \ref{specpoint}).

We also remark that if $\kappa = 0$ and $M(z)= i$ for all $z\in \bbC_+$, then the point spectrum of
the dissipative operator $\widehat \cB$ fills in the whole
open upper half-plane $\bbC_+$.
\end{remark}

Given a triple $(\dot A, \widehat A, A)$
satisfying  Hypothesis \ref{setup},  the following
corollary provides an analog of the 
Krein formula for resolvents
 for  all quasi-self-adjoint
dissipative extensions of the symmetric operator $\dot A$
with deficiency indices $(1,1)$ 
(cf., \cite{MMog}, where in the framework of the
 boundary triplets theory
the general  resolvent formula for dual pairs of linear
 relations was obtained). 

\begin{corollary}Under  Hypothesis \ref{setup},
 the following resolvent formula
 \begin{equation}\label{krres}
(\widehat A-zI)^{-1}=(A-zI)^{-1}-p(z)(\cdot\, ,g_{\overline{z}})g_z,
\end{equation}
$$z\in \rho(\widehat A)\cap \rho(A),
 $$
holds.

Here
\begin{itemize}
\item[(a)]
\begin{align}
 p(z)&=\left (M(\dot A,
A)(z)+i\frac{\kappa+1}{\kappa-1}\right )^{-1}
\label{WTf}\\&
=i\left (\frac{s(\dot A,A)(z)+1}{s(\dot A,A)(z)-1}
-\frac{\kappa+1}{\kappa-1}\right )^{-1};
\label{Lf}\end{align}
\item[(b)]$M(\dot A, A)$ and $s(\dot A, A)$ are
 the Weyl-Titchmarsh
and   the Liv\v{s}ic function
of the pair $(\dot A, A)$, respectively;
\item[(c)]
$g_z$ are deficiency elements of $\dot A$,
$$
g_z\in \Ker((\dot A)^*-zI),
$$
 satisfying the normalization condition
\begin{equation}\label{norcon}
\|g_z\|=\left (\int_\bbR \frac{d\mu(\lambda)}{|\lambda-z|^2}
\right )^{1/2}
\end{equation}
(the deficiency elements $g_z$ can be chosen to be analytic in
$z\in  \rho(\widehat A)\cap \rho(A)$);
\item[(d)]
 $\mu$ is the measure from the Herglotz-Nevanlinna representation
$$M(\dot A, A)(z)=\int_\bbR \left
(\frac{1}{\lambda-z}-\frac{\lambda}{1+\lambda^2}\right )
d\mu;
$$

and
\item[(e)] $\kappa$ is the von Neumann
 parameter characterizing the domain of
 the dissipative extension $\widehat A$,
\begin{equation}\label{vnp}
g_i-g_{-i}\in \Dom (A) \quad \text{and}\quad g_i-\kappa g_{-i}\in
\Dom (\widehat A).
\end{equation}
\end{itemize}
\end{corollary}

\begin{remark} We would like to stress 
 that 
\begin{itemize}
\item[i)] the von Neumann parameter
$\kappa$ from \eqref{vnp},
\item[ii)]
 the Liv\v{s}ic function $s(\dot A, A)$, 
 
and
\item[iii)]
the Weyl-Titchmarsh function
$M(\dot A, A)$,
\end{itemize}
 can easily  be recovered from the functional parameter of 
 the (prime) triple, the 
characteristic function $S=S(\dot A, \widehat A, A)$, 
which is a complete unitary invariant of $(\dot A, \widehat A, A)$. 

Indeed, 
\begin{align*}
\kappa  &=S(i),
\\s(\dot A, A)(z)&=\frac{S(z)-\kappa}{\overline{\kappa} S(z)-1}, 
\\
 M(\dot A, A)(z)&=\frac1i\cdot \frac{s(\dot A, A)(z)+1}{s(\dot A, A)(z)-1},
\end{align*}
$$z\in \bbC_+,
$$
with $ M(\dot A, A)$
continued to the lower half-plane by the Schwarz reflection principle
$$
M(\dot A, A)(z)=\overline{M(\dot A, A)(\overline{z})},\quad z\in \bbC_-.
$$
\end{remark}
\begin{remark} The resolvent formula \eqref{krres}--\eqref{Lf}
also holds if $|\kappa|=1$ and hence $\widehat A$ is self-adjoint. In this case, it
 provides the standard  Krein resolvent formula
for self-adjoint extensions of $\dot A$ via the von Neumann extension
 parameter $\kappa$
and the Weyl-Titchmarsh (Liv\v{s}ic)  function.

We also notice that although it appears unlikely 
that the explicit normalization condition \eqref{norcon} 
has been missed in the literature, we were not able to locate an appropriate reference.
\end{remark}

\begin{remark}
We  remark that if two triples
$(\dot A, A ,\widehat A_1)$ and $(\dot A, A ,\widehat A_2)$ satisfy Hypothesis
\ref{setup}
with the von Neumann parameters $\kappa_1$ and $\kappa_2$, respectively,
then one gets the following resolvent formula
 for the dissipative extensions $\widehat A_1$ and $\widehat A_2$
refining, in the rank one setting,  a result in \cite{K}:
$$
(\widehat A_2-zI)^{-1}=(\widehat A_1-zI)^{-1}-q(z)(\cdot, g_{\overline{z}})g_z,
$$
where
$
q(z)=p_2(z)-p_1(z)
$
with
\begin{align*}
p_k(z)&=
\left (M(\dot A,
A)(z)+i\frac{\kappa_k+1}{\kappa_k-1}\right )^{-1},
\\&=i\left (\frac{s(\dot A,A)(z)+1}{s(\dot A,A)(z)-1}
-\frac{\kappa_k+1}{\kappa_k-1}\right )^{-1},\quad k=1,2,
\end{align*}
$$
z\in \rho(A)\cap \rho(\widehat A_1)\cap \rho(\widehat A_2).
$$
We recall that if $S_1$ and $S_2$ are the characteristic functions of
the triples $(\dot A, A ,\widehat A_1)$ and $(\dot A, A ,\widehat A_2)$,
respectively, then
$$
s(\dot A,A)=\frac{S_k-\kappa_k}{\overline{\kappa_k} S-1},
\quad k=1,2.$$
\end{remark}

\appendix
\section{The spectral properties of the model symmetric operator}

 For completeness, we provide a summary of
 spectral properties of the model symmetric operator $\dot \cB $
in $L^2(\bbR; d\mu)$
(we refer to  a relevant  discussion in   \cite{GT})
and obtain
  the characterization of the core of its
 spectrum in terms of the measure $\mu$.

Denote by $\fm$ the class of infinite Lebesgue-Stieltjes measures
$\mu$ on the real axis, $\mu(\bbR)=\infty $, such that
 \begin{equation}\label{norm11}
\int_\bbR\frac{d\mu(\lambda)}{1+\lambda^2}=1.
\end{equation}

Our main goal is to show that every measure $\mu$ from the measure
class $\fm$ gives rise to a
  prime symmetric
operator in the Hilbert space $L^2(\bbR; d\mu)$ with deficiency
indices $(1,1)$ via the following construction.

Given $\mu\in \fm$, in the Hilbert space
$L^2(\bbR;d\mu)$ introduce
  the multiplication (self-adjoint)
operator $\cB$ by the  independent variable
 on
\begin{equation}\label{nachaloo}
\Dom(\cB)=\left \{f\in \,L^2(\bbR;d\mu) \,\bigg | \, \int_\bbR
\lambda^2 | f(\lambda)|^2d \mu(\lambda)<\infty \right \}
\end{equation} and denote by  $\dot \cB$  its
 restriction
on
\begin{equation}\label{nachaloA}
\Dom(\dot \cB)=\left \{f\in \Dom(\cB)\, \bigg | \, \int_\bbR
f(\lambda)d \mu(\lambda) =0\right \}.
\end{equation}

 We remark that the membership  $f\in \Dom(\cB)$ means that
$$
\int_\bbR (1+\lambda^2) |f(\lambda)|^2d\mu(\lambda)<\infty.
$$
Therefore, using Cauchy-Schwarz,
\begin{align}
\int_\bbR |f(\lambda)|d\mu(\lambda) &\le \left (\int_\bbR
(1+\lambda^2) |f(\lambda)|^2d\mu(\lambda)\right )^{1/2} \left
(\int_\bbR\frac{d\mu(\lambda)}{1+\lambda^2}\right )^{1/2}
\nonumber \\ &= \left (\int_\bbR (1+\lambda^2)
|f(\lambda)|^2d\mu(\lambda)\right )^{1/2}<\infty.
\label{welld}\end{align}

Inequality \eqref{welld} shows that  the following unbounded
 functional $\ell$ on
$ \Dom(\ell)=\Dom(\cB) $ given by
$$
\ell(f) = \int_\bbR f(\lambda)d \mu(\lambda), \quad f\in \dom(\ell),
$$
is well defined and hence the restriction
  $\dot \cB$ on $\Dom(\dot \cB)$ given by \eqref{nachaloA} is well defined
as well.

\begin{theorem}\label{modsym}
Suppose  that $(\dot \cB,\cB)$ is the model pair in the Hilbert
space $L^2(\bbR;d\mu)$ as
 defined by \eqref{nachaloo} and \eqref{nachaloA}.
Then $\dot \cB$ is a prime, densely defined, closed, symmetric
operator with deficiency indices $(1,1)$.
\end{theorem}
\begin{proof}

Since $\dot \cB$ is a restriction of a self-adjoint operator, the
operator $\dot \cB$ is symmetric.

To show that $\dot \cB$ is closed, assume that
$\{f_n\}_{n=1}^\infty$ is a sequence such that $f_n\in \Dom(\dot
\cB)$ for all $n$ and
$$
\lim_{n\to \infty}f_n=f\quad \text{ and } \quad \lim_{n\to
\infty}\dot \cB f_n=g
$$
for some $f, g\in L^2(\bbR; d\mu)$.  Since $\dot \cB$ is a
restriction of the
 self-adjoint operator $\cB$, one gets that
$$
\lim_{n\to \infty}f_n=f\quad \text{ and } \quad \lim_{n\to \infty}
\cB f_n=g
$$
and hence $f\in \Dom(\cB)$ and $g =\cB f$ for $\cB$ is a closed
operator.

Since $f_n\in \Dom(\dot \cB)$, one gets that
$$
\int_\bbR f_n(\lambda)d\mu(\lambda)= \int_\bbR
(\lambda-i)f_n(\lambda)\overline{\frac{1}{\lambda +i}}
d\mu(\lambda)=0
$$
which means that $(\cB-iI)f_n$ is orthogonal to
 the element $h\in L^2(\bbR; d\mu)$ given by
$$
h(\lambda)=\frac{1}{\lambda +i}.
$$
Therefore,
$$
\lim_{n\to \infty}((\cB-iI)f_n, h)=((\cB-iI)f, h)=0
$$
which means that $\int_\bbR f(\lambda)d\mu(\lambda)=0$ and hence
$f\in \Dom(\dot \cB)$.

To prove that $\dot \cB$ is densely defined, we observe first,
that the quotient space $\Dom(\cB)/\Dom(\dot \cB)$ is
one-dimensional. More specifically,
\begin{equation}\label{onedim}
\Dom(\cB)=\Dom(\dot \cB)\dot +\text{span}\left \{g\right \},
\end{equation}
where the function $g$ is given by
$$g(\lambda)=\frac{1}{1+\lambda^2},\quad \text{$\mu$-a.e.}\,.
$$
Indeed, for any $f\in \Dom(\cB)$, the function
$$
 g(\lambda)=f(\lambda)-\frac{\int_\bbR f(s)d\mu(s)}{1+\lambda^2}
$$
belongs to $\Dom(\dot \cB)$, since
$$
\int_\bbR g(\lambda)d\mu(\lambda)= \int_\bbR
f(\lambda)d\mu(\lambda)-\int_\bbR f(s)d\mu(s)\int_\bbR
\frac{d\mu(\lambda)}{1+\lambda^2}=0,
$$
due to the normalization
condition
$$
\int_\bbR\frac{d\mu(\lambda)}{1+\lambda^2}=1.
$$

To show that $\dom(\dot \cB)$ is dense in $L^2(\bbR; d\mu)$,
assume that $h\perp \Dom(\dot \cB)$.

Given $n\in \bbN$, introduce the function
$$
h_n(\lambda)=h(\lambda) \chi_{(-n,n)}(\lambda),
$$
where $  \chi_{(-n,n)}(\cdot )$ is the
 characteristic function of the interval $(-n,n)$.

The function
$$
g_n(\lambda, z)=\frac{h_n(\lambda)}{\lambda -z}-
\frac{1}{1+\lambda^2}\int_\bbR\frac{h_n(s)}{s-z}d\mu(s)
$$
 obviously belongs to $\dom(\dot \cB)$. In particular,
$$
g_n(\cdot, z)\perp h, \quad n\in \bbN,
$$
which means that
\begin{align*}
\int_\bbR \frac{h_n(\lambda)\overline{h(\lambda)}}{\lambda
-z}\,d\mu(\lambda)
&=\int_\bbR\frac{\overline{h(\lambda)}}
{1+\lambda^2}\,d\mu(\lambda)\cdot \int_\bbR\frac{h_n(s)}{s-z}\,d\mu(s)
\\&
=\int_\bbR\frac{h_n(\lambda)\overline{m}}{\lambda-z}\,d\mu(\lambda).
\end{align*}
Here we used the notation
$$ m=\int_\bbR\frac{{h(\lambda)}} {1+\lambda^2}d\mu(\lambda).
$$

By the uniqueness theorem for Cauchy integrals, one gets that
$$
h_n(\lambda)\overline{h(\lambda)}=|h(\lambda)|^2=h(\lambda)\overline{m}
\quad \text{ for }
 \mu\text{-a.e. }\lambda\in (-n,n).
$$

Therefore,
$$
h(\lambda)=m\quad \text{ for }
  \mu\text{-a.e. }\lambda\in (-n,n).
$$
Since $n$ is arbitrary, one concludes that
$$
h(\lambda)=m\quad
 \mu\text{-a.e. }.
$$
According to the  hypothesis, the measure $\mu$ is infinite. Therefore, $h\in L^2(\bbR;
d\mu)$ only if the constant $m$ is zero. That is,  $h=0$ which
proves that $\Dom(\dot \cB)$ is dense in $L^2(\bbR; d\mu)$, and hence,
 the operator $\dot \cB$ is densely defined.

Next, we prove that $\dot \cB$ is a prime operator.

Assume that $\cH_0$ reduces $\dot \cB$ and that the part $\dot
\cB|_{\cH_0}$ is a self-adjoint operator. Since  the  self-adjoint
multiplication  operator $\cB$ in $\cH=L^2(\bbR;d\mu)$ is an
extension of the symmetric operator $\dot \cB$, the subspace
$\cH_0$ also reduces the self-adjoint operator $\cB$ and hence
$$
\cH_0=\Ran E_\cB(\delta) \quad \text{for some Borel set }
\delta\subset \bbR
$$
for $\cB$ has a simple spectrum. Here $E_\cB(\cdot)$ denotes the projection-valued
 spectral measure of the self-adjoint operator $\cB$.

One observes that the function $f(\cdot )=\chi_{\delta\cap
(-n,n)}(\cdot)$ belongs to $\Dom(\cB|_{\cH_0})\subset \Dom(\dot
\cB)$. In particular,
$$
\int\limits_{(-n,n)\,\cap \,\delta}d\mu(\lambda)=\mu\left ((-n,n)\cap
\delta\right )=0\quad \text{ for all }n\in \bbN.
$$
Hence, $\mu (\delta)=0$ and therefore, $\cH_0=0$, proving that
$\dot \cB$ is a prime symmetric operator.

Finally, we  prove that $\dot \cB$ has deficiency indices
$(1,1)$.

 We claim that
$$
\Ker((\dot \cB)^*\mp iI)=\text{span}\{g_\pm\},
$$
where
$$
g_\pm(\lambda)=\frac{1}{\lambda\mp i}, \quad \text{$\mu$-a.e. }.
$$

First, we prove the inclusion
$$
\Ker((\dot \cB)^*\mp iI)\subset\text{span}\{g_\pm\}.
$$
Indeed, it suffices to show that $( (\dot \cB\pm iI)f,g_\pm)=0$
for all $f\in \Dom(\dot \cB)$. One computes
$$
(\dot \cB\pm iI)f,g_\pm)= \int_\bbR (\lambda\pm
i)f(\lambda)\overline{\frac{1}{\lambda\mp i}}d\mu(\lambda)
=\int_\bbR f(\lambda)d\mu(\lambda)=0,
$$
which shows that $\dot \cB$ has deficiency indices at least
$(1,1)$. Now the claim follows from \eqref{onedim} and the
observation that
$$
\frac{1}{2i}(g_+(\lambda)-g_-(\lambda))=\frac{1}{\lambda^2+1}
,\quad \text{$\mu$-a.e.}\,\,.$$
\end{proof}

The following  result characterizes the core of the spectrum of
the symmetric operator $\dot \cB$ as the set of non-isolated
points of the support of the measure $\mu$.

Recall that
 a point $\lambda\in \bbC$ is said to be  quasi-regular for  an operator $T$
 if
$T-\lambda I$ has a continuous inverse on $\Ran(T-\lambda I)$ and
the
 core of the spectrum of $T$ is defined to be the complement
to the set of its quasi-regular points.

\begin{theorem}\label{speccore} Assume that $\dot \cB$ is a
symmetric operator from the model pair in the Hilbert space
$L^2(\bbR,d\mu)$.

Then
 a  point $\lambda_0$, $\lambda_0\in \bbR$,
is a quasi-regular point for the symmetric operator $\dot \cB$ if
and only if there exists an $\varepsilon >0$ such that
\begin{equation}\label{critt}
\mu\left((-\varepsilon+\lambda_0, \lambda_0) \cup (\lambda_0,
\lambda_0+\varepsilon)\right )=0.
\end{equation}
\end{theorem}
\begin{proof} {\it ``Only If'' } Part. Suppose that \eqref{critt} holds true.
If $\mu(\{\lambda_0\})=0$, then the point $\lambda_0$ belongs to
the resolvent set of the
 self-adjoint operator $\cB$ and hence $\lambda_0$ is automatically a
quasi-regular point for the symmetric restriction $\dot \cB$ of
$\cB$.

If $\mu(\{\lambda_0\})>0$, one proceeds as follows.

 Suppose that
\begin{equation}\label{poslab}
u\in \Dom(\dot \cB).
\end{equation} Then
\begin{align*}
\|u\|^2&=\int_\bbR |u(\lambda)|^2d\mu(\lambda)=
\int\limits_{|\lambda-\lambda_0|\ge \varepsilon}
|u(\lambda)|^2d\mu(\lambda) +|u(\lambda_0)|^2 \mu(\{\lambda_0\})
\\ &\le
\left (\int\limits_{|\lambda-\lambda_0|\ge \varepsilon}
\frac{|u(\lambda)|^2}{(\lambda-\lambda_0)^2}d\mu(\lambda)\right
)^{1/2} \left (\int\limits_{|\lambda-\lambda_0|\ge
\varepsilon}(\lambda-\lambda_0)^2 |u(\lambda)|^2d\mu(\lambda)
\right )^{1/2}
\\ &+|u(\lambda_0)|^2 \mu(\{\lambda_0\})
\le \varepsilon^{-1}\|u\| \, \|(\dot
\cB-\lambda_0I)u\|+|u(\lambda_0)|^2 \mu(\{\lambda_0\}).
\end{align*}
Since \eqref{poslab} implies
$$
\int\limits_{|\lambda-\lambda_0|\ge \varepsilon}u(\lambda)d\mu(\lambda)+
u(\lambda_0)\mu(\{\lambda_0\})=0,
$$
and hence
$$
u(\lambda_0)=-(\mu(\{\lambda_0\}))^{-1}
\int\limits_{|\lambda-\lambda_0|\ge \varepsilon}u(\lambda)d\mu(\lambda),
$$
one gets the inequality
\begin{align*}
\|u\|^2&\le \varepsilon^{-1}\|u\| \, \|(\dot \cB-\lambda_0I)u\|
+
( \mu(\{\lambda_0\})^{-1}\left (\int\limits_{|\lambda-\lambda_0|\ge
\varepsilon}u(\lambda)d\mu(\lambda)\right )^2
\\&
\le \varepsilon^{-1}\|u\| \, \|(\dot \cB-\lambda_0I)u\|+  (\mu(\{\lambda_0\})^{-1}
\int\limits_{|\lambda-\lambda_0|\ge \varepsilon} \frac{d
\mu(\lambda)}{(\lambda-\lambda_0)^2}\|(\dot
\cB-\lambda_0I)u\|^2.
\end{align*}
Thus, the ratio
$$
r=\frac{\|u\|}{\|(\dot \cB-\lambda_0 I)u\|}
$$
satisfies the inequality
$$
r^2\le \varepsilon^{-1} r+(\mu(\{\lambda_0\}))^{-1}
\int\limits_{|\lambda-\lambda_0|\ge \varepsilon} \frac{d
\mu(\lambda)}{(\lambda-\lambda_0)^2},
$$
and therefore $r$ is uniformly bounded with respect to
$\varepsilon$ which proves that the point $\lambda_0$ is a
quasi-regular point.

{\it ``If"}  Part. If \eqref{critt} does not hold,
 and therefore $\lambda_0$ is not an isolated point from
the support of the measure $\mu$, we proceed as follows.  

Take an
$\varepsilon >0$ and choose
 measurable sets $\delta_\pm$ such that
$$
\delta_+\cup\delta_-=(\lambda_0-\varepsilon,
\lambda_0+\varepsilon)\setminus \{\lambda_0\},
$$
with $ \mu(\delta_\pm )>0. $ Set
$$
u_\varepsilon
(\lambda)=(\mu(\delta_+))^{-1}\chi_{\delta_+}(\lambda)
-(\mu(\delta_-))^{-1}\chi_{\delta_-}(\lambda).
$$
Clearly,  $u_\varepsilon\in L^2(\bbR;d\mu)$  and
$$
\int_\bbR u_\varepsilon(\lambda)d\lambda=0.
$$
Thus, $ u_\varepsilon \in \Dom (\dot \cB). $ One computes that
$$
\|u_\varepsilon\|^2=(\mu(\delta_+))^{-1}+(\mu(\delta_-))^{-1}.
$$
However,
$$
\|(\dot \cB-\lambda_0 I)u_\varepsilon\|^2\le \varepsilon^2\left (
(\mu(\delta_+))^{-1}+(\mu(\delta_-))^{-1}\right )=\varepsilon
^2\|u_\varepsilon\|^2
$$
which proves that $\lambda_0$ belongs to the core of  the spectrum
of $\dot A$ since $\varepsilon$ can be chosen  arbitrarily small.

The proof is complete.
\end{proof}

%%%%%%%%%%%%%%%%%%%%%%%%%%%%%%%%%%%%%%%%%%%%%%%%%%%%%%%%%%%%%%%%%%%%%%
%%%%%%%%%%%%%%%%%%%%%%%%%%%%%%%%%%%%%%%%%%%%%%%%%%%%%%%%%%%%%%%%%%%%%%%
%%%%%%%%%%%%%%%%%%%%%%%%%%%%%%%%%%%%%%%%%%%%%%%%%%%%%%%%%%%%%%%%%%%%%%%
%%%%%%%%%%%%%%%%%%%%%%%%%%%%%%%%%%%%%%%%%%%%%%%%%%%%%%%%%%%%%%%%%%%%%%%
%%%%%%%%%%%%%%%%%%%%%%%%%%%%%%%%%%%%%%%%%%%%%%%%%%%%%%%%%%%%%%%%%%%%%%%
%%%%%%%%%%%%%%%%%%%%%%%%%%%%%%%%%%%%%%%%%%%%%%%%%%%%%%%%%%%%%%%%%%%%%%
%%%%%%%%%%%%%%%%%%%%%%%%%%%%%%%%%%%%%%%%%%%%%%%%%%%%%%%%%%%%%%%%%%%%%%
%%%%%%%%%%%%%%%%%%%%%%%%%%%%%%%%%%%%%%%%%%%%%%%%%%%%%%%%%%%%%%%%%%%%%%
%%%%%%%%%%%%%%%%%%%%%%%%%%%%%%%%%%%%%%%%%%%%%%%%%%%%%%%%%%%%%%%%%%

\end{document}